\documentclass{article}

\usepackage{arxiv}

\usepackage[utf8]{inputenc} % allow utf-8 input
\usepackage[T1]{fontenc}    % use 8-bit T1 fonts
\usepackage{url}            % simple URL typesetting
\usepackage{booktabs}       % professional-quality tables
\usepackage{amsfonts}       % blackboard math symbols
\usepackage{nicefrac}       % compact symbols for 1/2, etc.
\usepackage{microtype}      % microtypography
\usepackage{lipsum}

\usepackage{relsize,balance,lipsum,bbm,enumerate,times,comment,color,graphicx,setspace,mathdots,mathrsfs,amssymb,latexsym,amsfonts,amsmath,cite,stmaryrd,caption,pgf,accents,mathtools,tabu,enumitem,hhline,array,epstopdf,nicefrac,amsthm,microtype,algorithmic,array,float,bm,url}
\newtheorem{theorem}{Theorem}
\newtheorem{definition}{Definition}

\newtheorem{lemma}{Lemma}
\newtheorem*{remark}{Remark}
\newtheorem{corollary}{Corollary}
\newtheorem{assumption}{Assumption}
\usepackage{graphicx}
\usepackage{mathtools} 

\usepackage[utf8]{inputenc}
\usepackage[english]{babel}

\DeclareMathOperator*{\argmin}{arg\,min}
\usepackage{tikz}

\usetikzlibrary{positioning}
\definecolor{mygreen}{RGB}{153,255,153}
\definecolor{myorange}{RGB}{255,178,102}
\definecolor{myred}{RGB}{255,153,153}
\definecolor{myblue}{RGB}{153,204,255}

\title{Non-Monotone Variational Inequalities}

\author{
    Sina Arefizadeh$^*$\\
    Dept. of Electrical and Computer Engineering\\
    Arizona State University\\
    Tempe, Arizona, USA\\
    \texttt{sarefiza@asu.edu}
\And
    Angelia~Nedi\'c\\
    Dept. of Electrical and Computer Engineering\\
    Arizona State University\\
    Tempe, Arizona, USA\\
    \texttt{Angelia.Nedich@asu.edu}
}

\theoremstyle{definition}
\begin{document}
\thanks{Corresponding author}
\thanks{ This work has been supported by the NSF award CCF 2106336.}
\maketitle
\begin{abstract}
In this paper, we focus on deriving some sufficient conditions for the existence of solutions to non-monotone Variational Inequalities (VIs) based on inverse mapping theory. We have obtained several widely applicable sufficient conditions for this problem and have introduced a sufficient condition for the existence of a Minty solution. We have shown that the extra-gradient method converges to a solution of VI in the presence of a Minty solution. Additionally, we have shown that, under some extra assumption, the algorithm is efficient and approaches a particular type of Minty solution. Interpreting these results in an equivalent game theory problem, weak coupling conditions will be obtained, stating that if the players' cost functions are sufficiently weakly coupled, the game has a pure quasi-Nash equilibrium. Moreover, under the additional assumption of the existence of Minty solutions, a pure Nash equilibrium exists for the corresponding game. 
\end{abstract}

%%%%%%%%%%%%%%%%%%%%%%%%%%
\section{Introduction}
%%%%%%%%%%%%%%%%%%%%%%%%%%
The study of Variational Inequalities (VIs) has recently become a topic of a wide interest. In particular, many solution concepts in game theory are tightly connected with solution concepts related to variational inequalities. The use of variational inequalities in game theory, particularly in equilibrium problems and their reformulations in the variational inequality framework, has been widely discussed in \cite{C6}. The role of Minty variational inequalities in studying evolutionary games' stable state in nonlinear population games is investigated in \cite{C5}. Moreover, there is a line of research, such as \cite{C7} and references therein, focusing on stochastic variational inequalities to model the equilibrium of stochastic games under uncertainty. Recently in \cite{C8}, the problem of guaranteeing the existence of Nash equilibrium in a class of network games is studied under sufficient conditions in the form of monotonicity of the game mapping.

%Concentration in a variety of studies, as mentioned above, is on monotonic variational inequalities, as most of the advancements in the related literature are to address the problems that arise in this class. 
The majority of research papers are devoted to designing algorithms to find a solution to monotone variational inequalities. Among these studies, one can refer to the first-order projection methods, which were first developed by studies on the gradient method by Sibony \cite{C9}, the proximal method by Martinent \cite{C10}, the extra gradient method by Korpelevich \cite{C4} and, later on, by studies such as modified forward-backward \cite{C11}, mirror-prox \cite{C12}, dual exploration \cite{C13}, hybrid proximal extra gradient \cite{C14} methods etc.

More recently, studies have also been on developing higher-order projection methods and determining their global convergence iteration complexities, such as in \cite{C15}, \cite{C16}, and \cite{C17}. For non-monotone VIs, there are studies focusing on deriving algorithms converging to a solution of the VI based on the existence of a Minty solution. The work in \cite{C2} and the references therein are examples of such studies. The study of sufficient conditions for the existence of a solution to a non-monotone VI itself is an interesting question that can shed some light on the problem of existence of a quasi Nash equilibrium in game theory. In the Generalized Nash Equilibrium Problem (GNEP), where the main goal is to guarantee the existence of a pure Nash equilibrium for a given game, there are three main assumptions that can be relaxed and become a direction for further studies. 

The literature considers the following directions for potential relaxations:
\begin{itemize}
\item [a)] Relaxation of the continuity of cost functions.
\item [b)] Relaxation of the compactness of action spaces.
\item [c)] Relaxation of the quasi-convexity of cost functions.
\end{itemize}
These three directions are not typically studied intensively all at once. However, there are extensive studies dedicated to each one individually. Among the studies on item a) in the list as mentioned above, we can refer to \cite{morgan2004existence}. This study does not address items b) and c). For item b), which can be alternatively studied through VIs by relaxing compactness of the game mapping domain, a relevant study is \cite{C19}. In this case for game mapping $F$ extensive literature examines the conditions under which a solution for the corresponding variational inequality (VI) exists. Additionally, some papers discuss potential relaxations of the action space's compactness, concluding with conditions on the compactness of lower-level sets of the Nikaido-Isoda function of the game \cite{C18}, \cite{C19}, \cite{C20}. 
Additionally, the quasi-convexity assumption, item c), is considered an artificial addition extraneous to the fundamental nature of the model, making its relaxation a reasonable consideration. Due to the complexity of such relaxation, the literature on this topic is limited. Among the few studies available, \cite{C18} is noteworthy; however, it does not address other types of relaxations.

Therefore, the main goals and contributions of this current study are threefold. Firstly, to study both items b) and c) of the above list simultaneously and to focus on the research question concerning the conditions under which a solution for a non-monotone VI in the presence of a nonempty closed convex but potentially non compact domain exists. Secondly, to guarantee the existence of a stronger version, a Minty solution for such a problem. Finally, to obtain a solution for non-monotone VIs numerically, which we accomplished by considering the extra-gradient method.

The organization of the rest of the paper is as follows: Section \ref{Sec-Notions & Terminology} is devoted to notions and terminology. 
%where the crucial definitions are discussed after the terminology is widely introduced. 
Section \ref{Sec-Main Results} is dedicated to the main results of the paper regarding non-monotone variational inequalities and how they can be interpreted in game theory. 
Section \ref{Sec-Algorithm analysis} considers Korpelevich method under a set of assumptions introduced in the preceding sections. Finally, Section \ref{Sec-conclusion} concludes the paper with a summary of contributions.

%---------------------------------------------------
\section{Notions and Terminology} \label{Sec-Notions & Terminology}
%---------------------------------------------------
In this section, we provide some definitions and terminologies about the variational inequalities. 

\begin{definition}
[Solution of VI \cite{facchinei2003finite}] Given a set $K\subset\mathbb{R}^m$ and a mapping $F:K\to\mathbb{R}^m$, the variational inequality problem, denoted by
VI$(K,F)$, consists of determining
   a point $x^*\in K$ such that 
   $$\langle F(x^*),x-x^*\rangle\geq 0\qquad\hbox{for all $x\in K$}.$$ 
   \end{definition}
A point $x^*$ satisfying the preceding inequality is referred to as a (strong) solution to the variational inequality problem VI$(K,F)$. The set of all such solutions is denoted by SOL$(K,F)$.

Another concept of a solution to a VI($K,F)$ exists, known as weak or Minty solution, defined as follows. 
   \begin{definition}[Minty Solution \cite{crespi2005existence}]
   Given a VI$(K,F)$, a point $x^*\in K$  such that 
   $$\langle F(x),x-x^*\rangle\geq 0\qquad\hbox{for all $x\in K$},$$ is a Minty solution to VI$(K,F)$. 
\end{definition}
The set of all Minty solutions to VI$(K,F)$ is denoted by MSOL$(K,F)$.

A well-studied class of variational inequalities is monotone VIs. The following definition introduces it formally.  
\begin{definition}[Strongly Monotone and Monotone mappings \cite{C2}]
   Given a set $K\subseteq\mathbb{R}^m$ and a mapping $F:K\to\mathbb{R}^m$, the mapping $F$ is strongly monotone if for some $\mu>0$ and for all $x,y\in K$, \[\langle F(x)-F(y),x-y\rangle\geq \mu\|x-y\|^2.\]
    The mapping $F$ is monotone if the preceding relation holds with $\mu=0$.
\end{definition}

The following lemma gives a relation between these two solution concepts.
\begin{lemma}[Lemma 2.2 - Minty's Lemma \cite{C2}]\label{Lem-Minty Lemma}
    Let $K \subseteq \mathbb{R}^m$ be a non-empty closed convex set
    and let $F:K \to\mathbb{R}^m$ be a mapping. The following statements hold:
    \begin{itemize}
        \item[(a)]
        If $F$ is continuous, then
    every Minty solution to VI$(K,F)$ is also a solution to VI$(K,F)$,  i.e., 
    \[{\rm MSOL}(K,F)\subseteq {\rm SOL}(K,F).\]
    \item[(b)] If $F$ is monotone, then every solution to VI$(K,F)$ is also a Minty solution to VI$(K,F)$, i.e.,
    \[{\rm SOL}(K,F)\subseteq {\rm MSOL}(K,F).\]
    \end{itemize}
\end{lemma}
\begin{proof}  
(a) Let us consider an $x^*\in {\rm MSOL}(K,F)$. Hence, \begin{equation}\label{eq-minty lemma 1}
\langle F(x),x-x^*\rangle\geq 0\qquad\hbox{for all $x\in K$}.
   \end{equation}
For any arbitrary $x\in K$, consider the point $v=x^*+t(x-x^*)$ where $t\in (0,1]$. Note that $v\in K$ since $x,x^*\in K$ and $K$ is convex. Thus, by using $v\in K$,
from~\eqref{eq-minty lemma 1} we have
\begin{equation*}
%\label{eq-minty lemma 2}
    \langle F(x^*+t(x-x^*)),t(x-x^*)\rangle\geq 0\quad\hbox{for all $x\in K,t\in (0,1]$}.
    \end{equation*}
 Since $t>0$, it follows that
\begin{equation*}
%\label{eq-minty lemma 3}
    \langle F(x^*+t(x-x^*)),x-x^*\rangle\geq 0\qquad\hbox{for all $x\in K,t\in (0,1]$}.
    \end{equation*}
Letting $t$ go to zero and using the continuity of $F$ we have
%\begin{equation*}
%\label{eq-minty lemma 4}\langle F(x^*),x-x^*\rangle\geq 0\qquad\hbox{for all $x\in K$}.
%\end{equation*}
Hence, $x^*\in {\rm SOL}(K,F)$, implying that
\[{\rm MSOL}(K,F)\subseteq {\rm SOL}(K,F).\]
(b) Let $x^*$ be a solution to VI$(K,F)$.
By the monotonicity of $F$, we have
\begin{equation}\label{eq-minty lemma 5}
\langle F(x)-F(x^*),x-x^*\rangle\geq 0.
\end{equation}
Since $x^*\in {\rm SOL}(K,F)$ it follows that $\langle F(x^*),x-x^*\rangle\geq 0$. Therefore, from~\eqref{eq-minty lemma 5} it follows that $\langle F(x),x-x^*\rangle\geq 0$, thus implying that $x^*\in {\rm MSOL}(K,F)$. 
\end{proof}

Combining parts (a) and (b) of Lemma~\ref{Lem-Minty Lemma}, we see that 
for a continuous and monotone mapping $F$
we have
\[{\rm MSOL}(K,F)= {\rm SOL}(K,F).\] 
This result has been shown in Lemma~1.5 of \cite{KSbook}, where both monotonicity and continuity of the mapping $F$ are assumed. Our Lemma~\ref{Lem-Minty Lemma} considers these properties separately to gain a deeper insight into the role of these properties in the relations among the solutions and Minty solutions of a VI$(K,F)$.

%---------------------------------------------------
\section{Main Results} \label{Sec-Main Results}
%---------------------------------------------------
In this section, we provide some sufficient conditions for the existence of solutions to a VI problem.
In the sequel, we will use the Inverse mapping Theorem, provided below.
We let $\nabla F(\cdot)$ denote the Jacobian of a mapping $F(\cdot)$. We denote the determinant of the Jacobian $\nabla F(x)$ by $|\nabla F(x)|$. 

\begin{theorem}[Inverse Mapping Theorem \cite{lebl2018introduction}]\label{Thm-Inverse function}
    Given a vector $\textbf{a}$, let $F(\cdot)$ be a continuously differentiable mapping on some open set containing the vector $\textbf{a}$. Suppose that $|\nabla F(\textbf{a})|\neq 0$. Then, there is an open set $V$ containing the vector $\textbf{a}$ and an open set $W$ containing the vector $F(\textbf{a})$ such that $F:V\rightarrow W$ has a continuous inverse mapping $F^{-1}:W\rightarrow V$, which is continuously differentiable for all $y\in W$.  
\end{theorem}

Defining $F(\mathbb{R}^{m}):=\{F(x)|x\in \mathbb{R}^m\}$, in the rest of the paper, we draw attention to the class of mappings $F$ with closed $F(\mathbb{R}^{m})$ 
 and the class of continuous mappings and study them. One example of mappings with  $F(\mathbb{R}^{m})$ being a closed set is the class of closed mappings as they are defined in the sequel.
\begin{definition}[Closed Mapping\cite{C1}]\label{Def-Close map}
A mapping $F:\mathbb{R}^m\rightarrow \mathbb{R}^m$ is closed if for every closed set $C\subseteq \mathbb{R}^m$, the image set $F(C)$ is closed, where $F(C)=\{F(x)\mid x\in C\}$.
\end{definition}

In what follows, we use $B_r(x)$ to denote an open ball centered at a point $x\in \mathbb{R}^m$ with a radius $r>0$.
The following result provides sufficient conditions guaranteeing the existence of solution for VI$(\mathbb{R}^m,F)$.
\begin{theorem} \label{Thm-main theorem}
    Let $F:\mathbb{R}^m\to\mathbb{R}^m$ be a continuously differentiable and $F(\mathbb{R}^{m})$ be closed such that 
    $|\nabla{F(x)}|\neq 0$
    for every $x\in \mathbb{R}^m$ where $F(x)\neq 0$. Then, the VI$(\mathbb{R}^m,F)$ has a solution.
\end{theorem}
\begin{proof}
    Let $b=\inf_{x\in \mathbb{R}^m} \|F(x)\|$ and let 
    $\{x^k\}_{k=1}^\infty$ be a sequence such that $\lim_{k\rightarrow \infty}\|F(x^k)\|=b$.
To arrive at a contradiction, assume that $b>0$. 

Since $\|F(x^k)\|\to b$,
it follows that $\{F(x^k)\}_{k=1}^\infty$ is bounded and has a convergent sub-sequence $\{F(x^{k_i})\}_{i=1}^\infty$ with $\lim_{i\rightarrow \infty}F(x^{k_i})=\bar{F}$, where $\|\bar{F}\|=b$. 
Since $\bar{F} \in F(\mathbb{R}^m)$ and $F(\mathbb{R}^m)$ is closed, there is some $\bar{x}\in \mathbb{R}^m$ such that $F(\bar{x})=\bar{F}$. Since $\|\bar{F}\|=b$ and $b>0$, we have that  $F(\bar{x})\ne 0$. By the assumption of the theorem, it follows that $|\nabla{F(\bar{x})}|\neq 0$. 

    By the Inverse mapping Theorem~\ref{Thm-Inverse function}, if $|\nabla{F(\bar x)}|\neq 0$, then there exist open balls $B_r(\bar x)$ and $B_{r'}(F(\bar x))$, and a local inverse mapping $F^{-1}_{\bar x}(\cdot):B_{r'}(F(\bar x))\to B_r(\bar x)$ such that $F^{-1}_{\bar x}(v)=u$ for all $v\in B_{r'}\left(F(\bar x)\right)$ and $u\in B_{r}(\bar x)$, where $F(u)=v$.
    Therefore, there exists $\alpha\in(0,1]$ such that $$(1-\alpha)F(\bar{x})\in B_{r'}\left(F(\bar{x})\right).$$ 
    Thus, by the inverse mapping theorem, we have that $F^{-1}_{\bar{x}}\left((1-\alpha)F(\bar{x})\right)=z$ for some $z\in B_{r}(\bar{x})$, such that $$F(z)=F\left(F^{-1}_{\bar{x}}\left((1-\alpha)F(\bar{x})\right)\right)=(1-\alpha)F(\bar{x}).$$ Hence $0\leq\|F(z)\|< \|F(\bar{x})\|=b$ with $z\in \mathbb{R}^m$ -- a contradiction in view of $b=\inf_{x\in \mathbb{R}^m} \|F(x)\|$. Therefore, $b=0$ and $F(\bar x)=0$, implying that $\bar{x}\in {\rm SOL} (\mathbb{R}^m,F)$.
\end{proof}

The next two lemmas provide sufficient conditions for $|\nabla{F(x)}| \neq 0$. In what follows, we use $[m]$ to denote the set $\{1,2,\ldots,m\}$ for an integer $m\ge1$.
Also, we use $\nabla_{x_i}f(x)$ to denote the partial derivative of a function $f$ with respect to variable $x_i$.

\begin{lemma}[Weak Coupling Condition]\label{Lem-Weak coupling}
Let $F:\mathbb{R}^m\to\mathbb{R}^m$ 
be a mapping given by
$F=(F_1,F_2,\ldots, F_m)$, where $F_i:\mathbb{R}^m\to\mathbb{R}$ is the $i$-th component of the mapping $F$ for all $i\in[m]$.
If at a given point $x$ it holds
\[\left|\nabla_{x_i} F_i(x)\right|>\sum_{\substack{j=1 \\ j \neq i}}^m\left|\nabla_{x_j}F_i(x)\right|
\qquad\hbox{for every $i\in [m]$},\]
 then $|\nabla{F(x)}| \neq 0$.
\end{lemma}
 
\begin{proof}   
Denoting $M_{ij}$ as the $ij$-th element of the square matrix $M$ of dimension $m$, let us define the Gershgorin disk $C_i\left(M_{ii}, \sum_{\substack{j=1 \\ j \neq i}}^m\left|M_{ij}\right|\right)\subset \mathbb{C}$ for the matrix $M$ as a disk centered at $M_{ii}$ with the radius of $\sum_{\substack{j=1, j \neq i}}^m\left|M_{ij}\right|$. By Gershgorin circle Theorem \cite{bell1965gershgorin}, every eigenvalue of the Jacobian $\nabla F(x)$ lies in one of the Gershgorin disks $C_i\left(\nabla_{x_i} F_i(x), \sum_{\substack{j=1 \\ j \neq i}}^m\left|\nabla_{x_j} F_i(x)\right|\right)$, $ i \in[m]$. Under the weak coupling condition, it follows that  $0$ is not in any of Gershgorin disks. Thus, $0$ is not an eigenvalue of $\nabla F(x)$, implying that $|\nabla{F(x)}| \neq 0$.
\end{proof} 

% \begin{remark}\label{Rem-Compare results}
%     In the case of differentiable $F$, one can show that mapping $F$ is monotone if and only if $\nabla{F(x)}$ is positive semi-definite \cite{C20}. Thus, it tells us the coupling behavior must somehow appear in the condition of existence of a solution for the $VI(\mathbb{R}^m,F)$. Based on this observation, the previous lemma seems quite natural for the case of a positive definite Jacobian. However, it covers some scenarios where the map $F$ is not monotone.

% \end{remark}

The following lemma shows that a wide range of non-monotone VIs have a solution in the unconstrained case. 

\begin{lemma}\label{Lem-Alternative to be non singular}
Let $F:\mathbb{R}^m\to\mathbb{R}^m.$ Then, for any $x\in\mathbb{R}^m$, we have    $\left|\nabla{F(x)}\right| \neq 0$ if and only if the matrix $\nabla{F(x)} \nabla{F(x)}^T$ is positive definite.
 \end{lemma}
 \begin{proof} We have that
$|\nabla{F(x)}|^2=|\nabla{F(x)} \nabla{F(x)}^T|$. Thus, the determinant $|\nabla{F(x)}|^2$ is zero if and only if $|\nabla{F(x)} \nabla{F(x)}^T|=0$. So, alternatively we can consider the matrix $\nabla{F(x)}\nabla{F(x)}^T$, which is symmetric. Hence, all the eigenvalues of $\nabla{F(x)} \nabla{F(x)}^T$ are real and the matrix is positive semi-definite due to its form.  Therefore, $\left|\nabla{F(x)}\right| \neq 0$ is equivalent to $\nabla{F(x)} \nabla{F(x)}^T$ being positive definite.
 \end{proof}

 We have the following result as an immediate consequence of Theorem~\ref{Thm-main theorem} and Lemma~\ref{Lem-Alternative to be non singular}.
 In the result, we use $I$ to denote the identity mapping.
\begin{corollary}\label{Cor-Orthogonal map}
If $\nabla{F(x)}^{-1}=\nabla{F(x)}^T$ for all $x\in\mathbb{R}^m$ (orthogonal mapping $F$) and $F(\mathbb{R}^{m})$ is closed,
then $\nabla F(x)\nabla F(x)^T=I$ and the
VI$\left(\mathbb{R}^m, F\right)$ has a solution.
\end{corollary}

\def\argmin{\mathop {\rm argmin}}

In some applications, we might be interested in finding a solution for the constrained VI problem, i.e., VI$(K,F)$ where $K$ is a nonempty, closed, and  convex set but not necessarily bounded. To proceed with such scenarios, we will use an alternative representation of the VI$(K,F)$ problem to relate it to finding a zero of a suitably defined mapping.  The natural mapping associated with a VI$(K,F)$ plays a key role in the reformulation, which is defined as follows:
\[F_K^{nat}(v)=v-\Pi_K\left[v-F(v)\right]
\qquad\hbox{for all $v\in\mathbb{R}^m$},\] where $F:K\to\mathbb{R}^m$ and $\Pi_K[\cdot]$ is the Euclidean projection on the closed convex set $K\subseteq\mathbb{R}^m$, i.e., $\Pi_K[z]=\argmin_{x\in K}\|x-z\|^2$.
The following theorem relates the solutions of a VI problem with the zeros of its associated natural mapping.
\begin{theorem}[Proposition 1.5.8 \cite{facchinei2003finite}]\label{Thm-auxiliary-natural map solution}
    Let $F:K\rightarrow \mathbb{R}^m$ be a mapping defined on a set $K\subseteq \mathbb{R}^m$ which is nonempty, closed and convex. Then, we have
    \begin{equation}
        [x^*\in {\rm SOL}(K,F)] \iff [F_K^{nat}(x^*)=0].
    \end{equation}
\end{theorem}

Depending on the set $K$, the mapping $F_K^{nat}(\cdot)$ may not be differentiable. However, 
due to the non-expansiveness property of the projection mapping, we can see that $F_K^{nat}(\cdot)$ is Lipschitz continuous whenever $F$ is 
Lipschitz continuious. Therefore, to obtain an alternative to Theorem~\ref{Thm-Inverse function} where the mapping differentiability assumption is relaxed, we will assume the Lipschitz continuity of $F$. To develop such a theorem, we first define several concepts in the sequel.
\begin{definition}[Generalized Jacobian, Definition 1 \cite{clarke1976inverse}] \label{Def-Generalized Jacobian}
The generalized Jacobian of a mapping $F$ at point $x_0\in \mathbb{R}^m$, denoted by $\partial F(x_0)$, is the convex hull of all matrices $M$ of the form of 
\begin{equation}
\label{eq-gen-jac}
    M=\lim_{i\rightarrow \infty} \nabla F(x_i),
\end{equation}
where $\lim_{i\to\infty}x_i=x_0$ and $F$ is differentiable at $x_i$ for all $i$.
\end{definition}

We use the following result for the generalized Jacobian.

\begin{theorem}[Proposition 1 \cite{clarke1976inverse}]\label{Thm-Proposition 1 of Clark paper}
    Let mapping $F$ be Lipschitz continuous in a neighborhood of a point $x_0\in \mathbb{R}^m$. Then, the generalized Jacobian $\partial F(x_0)$ is a nonempty, compact, and convex set.
\end{theorem}
\begin{definition}[Definition 2 \cite{clarke1976inverse}] \label{Def-fullrank Generalized Jacobian}
A generalized Jacobian
$\partial F(x_0)$ is said to be of maximal rank if every matrix $M$ in the definition of $\partial F(x_0)$ (see~\eqref{eq-gen-jac}) has a full rank.
\end{definition}
The following theorem, known as the Clark inverse mapping Theorem, is the key to extending our Theorem~\ref{Thm-main theorem}.

\begin{theorem}[Clark Inverse mapping - Theorem 1, \cite{clarke1976inverse}] \label{Thm-Clark inverse mapping theorem}
Let $F:\mathbb{R}^m\to\mathbb{R}^m$ be a mapping.
    Let $\partial F(x_0)$ be of a maximal rank for some $x_0\in\mathbb{R}^m$. Then, there exist neighborhoods $U$ and $V$ of $x_0$ and  $F(x_0)$, respectively, and a Lipschitz continuous mapping $G:V\to \mathbb{R}^m$ such that 
    \begin{itemize}
        \item [a)] $G\left(F(u)\right)=u$ for every $u\in U$,
        \item [b)] $F\left(G(v)\right)=v$ for every $v\in V$.
    \end{itemize}
\end{theorem}

The following corollary gives a sufficient condition to have a solution for VI$(K,F)$.

\begin{corollary}\label{Cor-Alternative to main theorem}
    Let set $K\subseteq\mathbb{R}^m$ be nonempty closed convex and let $F_K^{nat}(\mathbb{R}^m)$ be a closed set. Also, assume that $\partial F_K^{nat}(x)$ has a maximal rank for every $x\in \mathbb{R}^m$ where $F_K^{nat}(x)\neq 0$. Then, the VI$(K,F)$ has a solution.
\end{corollary}
\begin{proof}
    The proof follows from steps similar to that of the proof of Theorem~\ref{Thm-main theorem}, where we consider $F_K^{nat}(x)$ instead of $F(x)$ and use the Clark inverse mapping theorem instead of  the inverse mapping theorem. Also, we use the connection between the zeros of the natural mapping $F_K^{nat}(\cdot)$ and the solutions of the VI$(K,F)$ as given in Theorem~\ref{Thm-auxiliary-natural map solution}.
\end{proof}
%NOTE This result will be proven in the paperversion - not a corollary - a min result.

Using Definition \ref{Def-Close map} of a closed mapping, we can not immediately make a conclusion regarding the closedness of a mapping $F_K^{nat}$ given the fact that $F$ is a closed mapping, even closeness of $F(\mathbb{R}^{m})$ may not directly guarantee the closeness of $F_K^{nat}(\mathbb{R}^{m})$. In the sequel, we will explore sole continuity as an alternative to the closed mappings which can be guaranteed for the natural mapping of a given mapping considering non-expansive property of projection mapping. A continuous mapping need not be closed in the sense of Definition~\ref{Def-Close map}. 
% We start with a weaker notion of closedness related to the closed graph, as given in the following definition.

% \begin{definition}[Closed Graph - Definition 2.1.16, \cite{facchinei2003finite}]\label{Def-Closed graph}
%     A mapping $F:\mathbb{R}^m\rightarrow\mathbb{R}^m$ is said to be closed at a point $x^0\in\mathbb{R}^m$ if the following condition is satisfied: 
% \[\{x^k\}\rightarrow x^0\quad\hbox{and}\quad
% F(x^k)\rightarrow \bar{y}\qquad \implies
% \qquad \bar{y}=F(x^0).\]
% We say that $F:\mathbb{R}^m\to\mathbb{R}^m$ has a closed graph if $F$ is closed at every point $x^0\in\mathbb{R}^m$.
% \end{definition}

% Consider, for example, a one-dimensional map $$F(x)=\frac{1}{x}\qquad \hbox{for }x\in\mathbb{R}.$$
% This map has a closed graph, but it is neither continuous nor closed. Apparently, it is not continuous at $x=0$. Moreover, it is not closed since it maps the real line $\mathbb{R}$ (a closed set) into the open set $(-\infty,0)\cup(0,+\infty)$.

%results in closedness in the sense of Definition~\ref{Def-Closed graph} \an{- is this true?}. However, the converse statement is not necessarily true and requires some extra assumptions to hold true, such as the compactness of the space. 

To obtain an alternative to Theorem~\ref{Thm-main theorem} focusing only on its continuity, we need an additional assumption, as stated below.

\begin{assumption}\label{Assum-Semi coercivity}
Let $F:\mathbb{R}^m\to\mathbb{R}^m$ be such that,
for every sequence $\{x^k\}\subset\mathbb{R}^m$,
if
     $\{\|F(x^k)\|\}$ is bounded, then $\{\|x^k\|\}$ is also bounded.
 \end{assumption}

\begin{remark}\label{rem-coercive}
    Assumption~\ref{Assum-Semi coercivity}  is met when $\|F(x)\|$ is a coercive function of $x$, i.e., $\lim_{\|x\|\to\infty} \|F(x)\|=+\infty$.
\end{remark}

Now, we have the following theorem.
\begin{theorem} \label{Thm-main theorem-prime}
    Let $F:\mathbb{R}^m\rightarrow\mathbb{R}^m$ be continuously differentiable. Also, let Assumption~\ref{Assum-Semi coercivity} hold and assume that 
    $|\nabla{F(x)}|\neq 0$
    for every $x\in \mathbb{R}^m$ where $F(x)\neq 0$. Then,  the $VI(\mathbb{R}^m,F)$ has a solution.
\end{theorem}
\begin{proof}
Let $b=\inf_{x\in \mathbb{R}^m} \|F(x)\|$ and  let $\{x^k\}_{k=1}^\infty$ be a sequence such that $\lim_{k\rightarrow \infty}\|F(x^k)\|=b$. 
Thus, the sequence $\{F(x^k)\}_{k=1}^\infty$ is bounded and has a convergent sub-sequence $\{F(x^{k_i})\}_{i=1}^\infty$ with $\lim_{i\rightarrow \infty}F(x^{k_i})=\bar{F}$, where $\|\bar{F}\|=b$. Moreover, by Assumption~\ref{Assum-Semi coercivity}, the sequence $\{x^{k_i}\}_{i=1}^\infty$ is also bounded and, consequently, has a sub-sequence converging to some $\bar{x}$. Along this sub-sequence, the mapping values $F(x^{k_i})$ are also converging to $\bar{F}$. Without loss of generality, we may assume that $\lim_{i\to\infty}x^{k_i}=\bar x$ and  $\lim_{i\to\infty}F(x^{k_i})=\bar F$. By the continuity of $F$ it follows that $F(\bar x)=\bar F$, where $\|\bar F\|=b$.

To arrive at a contradiction, we assume that $b>0$. By the Inverse mapping Theorem (Theorem~\ref{Thm-Inverse function}), since $|\nabla{F(\bar x)}|\neq 0$, there are open balls $B_r(\bar x)$ and $B_{r'}\left(F(\bar x)\right)$, and a locally invertible mapping $F^{-1}_{\bar x}(v)$ (for the mapping $F$) such that $F^{-1}_{\bar x}(v)=u$ for all $v\in B_{r'}\left(F(\bar x)\right)$ and $u\in B_{r}(\bar x)$, where $F(u)=v$.
From now onward, the proof follows the same line of analysis as that of Theorem~\ref{Thm-main theorem}, leading to a contradiction that   
% As a result, considering the continuity of mapping $F$ and the fact that it is closed in every element of $\mathbb{R}^m$ in the sense of Definition~\ref{Def-Closed graph}, we have $\bar{F}=F(\bar{x})$.
%Since $b>0$, we have $\bar{F}\neq 0$. By statement of the theorem we have $|\nabla{F(\bar{x})}|\neq 0$. Therefore, there exists $\alpha$ such that $0<\alpha\leq 1$ such that $$(1-\alpha)F(\bar{x})\in N_{r'_{\bar{x}}}\left(F(\bar{x})\right)$$ for every $r'_{\bar{x}}>0$. Thus, there is (by the inverse mapping theorem) $F^{-1}_{\bar{x}}$ such that $F^{-1}_{\bar{x}}\left((1-\alpha)F(\bar{x})\right)=z\in N_{r_{\bar{x}}}(\bar{x})$, such that $$F(z)=F\left(F^{-1}_{\bar{x}}\left((1-\alpha)F(\bar{x})\right)\right)=(1-\alpha)F(\bar{x}).$$ 
$0\leq\|F(z)\|< \|F(\bar{x})\|=b$ for some $z\in \mathbb{R}^m$. Therefore,  we must have $b=0$ and $F(\bar x)=\bar{F}=0$, implying that $\bar{x}$ is a solution of the VI$(\mathbb{R}^m,F)$.
\end{proof}

The following theorem provides sufficient conditions for the existence of a Minty solution to VI$(K,F)$. We note that a VI$(K,F)$ with a closed convex set $K$ and strongly monotone mapping $F$ always has a unique solution (Theorem~2.3.3 in \cite{facchinei2003finite}).

\begin{theorem}\label{Thm-Minty sol}
Let set $K\subseteq\mathbb{R}^m$ be nonempty, closed, and convex, and let $\varphi:K\to\mathbb{R}^m$ be a strongly monotone mapping with $\Tilde{x}\in {\rm SOL}(K,\varphi)$.
Assume that $\|\varphi(x)-F(x)\|\le d\|x-\Tilde{x}\|$ for some $d <\mu_{\varphi}$ and for all $x\in K$, where $\mu_{\varphi}$ is the strong monotonicity constant of the mapping $\varphi$.
Then, $\Tilde{x}$ is a Minty solution to VI$(K, F)$.
\end{theorem}
\begin{proof}
 By the strong monotonicity of $\varphi$ and $\Tilde{x}\in {\rm SOL}(K,\varphi)$, we have
\begin{equation}\label{eq-var}
\langle\varphi(x), x-\Tilde{x}\rangle \geq \mu_{\varphi}\|x-\Tilde{x}\|^2\qquad\hbox{for all }x\in K.
\end{equation}
By the assumption that $\|\varphi(x)-F(x)\|\le d\|x-\Tilde{x}\|$ for all $x\in K$, it follows that
$$
\langle\varphi(x)-F(x), x-\Tilde{x}\rangle \le d\|x-\Tilde{x}\|^2\qquad\hbox{for all }x\in K.
$$
Let $L_<(\Tilde{x})=\{x\in K \mid\langle F(x), x-\Tilde{x}\rangle< 0\},$ 
and assume that $L_<(\Tilde{x})$ is nonempty. Then, we have
for all $x \in L_<(\Tilde{x})$, $$\langle\varphi(x), x-\Tilde{x}\rangle\le  d\|x-\Tilde{x}\|^2 < \mu_{\varphi}\|x-\Tilde{x}\|^2,$$
which contradicts relation~\eqref{eq-var} 
since $L_<(\Tilde{x})\subseteq K$. Thus, we must have $L_<(\Tilde{x})=\varnothing$, implying that
\[\left\langle F(x), x- \Tilde{x}\right\rangle \geq 0 \qquad\hbox{for all }x\in K .\]
Hence, $ \Tilde{x} \in {\rm MSOL}(K, F)$.
\end{proof}

\begin{remark}\label{Rem-Interpretation in games}
    In \cite{C20}, in the context of game theory for strongly convex cost functions in terms of the decision variable of the agents, in particular, where the Jacobian matrix is strictly diagonally dominant, it is proved that Nash equilibrium exists. This case can be considered as a particular case of Theorem~\ref{Thm-Minty sol}, where the mapping of the game is, in particular, strongly monotone, and the Jacobian has a strictly diagonally dominant structure. 
\end{remark}

\section{Extra-Gradient Method}\label{Sec-Algorithm analysis}
In this section, we consider the extra-gradient method (aka Korpelevich method)~\cite{C4} for the particular problem setup in the preceding section.
Specifically, while this method have been studied for monotone VIs, its convergence behavior for non-monotone VIs has not been thoroughly investigated. In our study, we use the following assumptions. 

\begin{assumption}\label{Assum-set-Lip}
    Let $K\subseteq\mathbb{R}^m$ be a nonempty closed convex set, and let the mapping  $F:K\rightarrow \mathbb{R}^m$  be Lipschitz continuous, i.e., 
    there exists a constant $L$ such that $\|F(x)-F(y)\|\leq L\|x-y\|$ for all $x,y\in K$.
\end{assumption}
%\begin{assumption}\label{Assum-Lipschitz continuity}
%There exists a constant $L$ such that $\|F(x)-F(y)\|\leq L\|x-y\|$ for all $x,y\in K$.
%\end{assumption}

The extra-gradient method is given by: for all $k\ge0$,
\begin{align}\label{Alg-Korpelevich}
        &y^k=\Pi_K[x^k-\alpha F(x^k)],\nonumber \\
        &x^{k+1}=\Pi_K[x^k-\alpha F(y^k)],
    \end{align}
    where $\alpha>0$ is a stepsize, and $x^0,y^0\in K$ are arbitrary initial points.
The following theorem shows that, having a Minty solution to VI$(K,F)$,  the extra-gradient method generates a bounded sequence $\{x^k\}$ with accumulation points in the set SOL$(K,F)$, for a suitable selection of the stepize.
\begin{theorem}[Extra Gradient Method \cite{C4}]\label{Thm-extragradient main variant}
    Let Assumption~\ref{Assum-set-Lip} hold. Assume that  the VI$(K,F)$ has a Minty solution, i.e., there is $\Tilde{x}\in {\rm MSOL}(K,F)$. Then, the sequence $\{x^k\}_{k=0}^\infty$ generated by the extra-gradient method~\eqref{Alg-Korpelevich}, with the stepsize $0<\alpha<\frac{1}{L}$, is bounded and every of its accumulation points  $\hat{x}$ is a solution to VI$(K,F)$.
    %, i.e., $\hat{x}\in {\rm SOL}(K,F)$.
\end{theorem}
\begin{proof}
     Using the properties of the projection, from the definition of the iterate $x^{k+1}$ we have that for all $k\ge0$ and any $x\in K$,
\begin{align}\label{eq-alg 1}
        \|x^{k+1}-x\|^2\leq&\|x^k-\alpha F(y^k)-x\|^2 
        -\|x^k-\alpha F(y^k)-x^{k+1}\|^2\cr
        =&\|x^{k}-x\|^2-\|x^{k}-x^{k+1}\|^2 
        +2\alpha\langle F(y^k),x-x^{k+1} \rangle.
    \end{align}
    Letting $x=\Tilde{x}$, where
    $\Tilde{x}$ is a Minty solution, since $y^k\in K$, we have $\langle F(y^k),\Tilde{x}-y^{k} \rangle \leq 0$ for all $k\ge0$. Therefore,\begin{align}\label{eq-alg 2}
        \langle F(y^k),\Tilde{x}-x^{k+1} \rangle&=\langle F(y^k),\Tilde{x}-y^{k} \rangle 
        +\langle F(y^k),y^{k}-x^{k+1} \rangle \nonumber \\
        &\leq \langle F(y^k),y^{k}-x^{k+1} \rangle.
    \end{align}
    Using \eqref{eq-alg 2} in \eqref{eq-alg 1}, where $x=\tilde x$, we can write
     \begin{align}\label{eq-alg 3}
        \|x^{k+1}-\Tilde{x}\|^2\leq&\|x^{k}-\Tilde{x}\|^2-\|x^{k}-x^{k+1}\|^2 \nonumber \\
        &+2\alpha\langle F(y^k),y^{k}-x^{k+1} \rangle \nonumber \\
        =&\|x^{k}-\Tilde{x}\|^2-\|x^{k}-y^{k}\|^2-\|y^{k}-x^{k+1}\|^2 \nonumber \\
        &-2\langle x^{k}-y^{k},y^{k}-x^{k+1} \rangle+2\alpha\langle F(y^k),y^{k}-x^{k+1} \rangle \nonumber \\
        =&\|x^{k}-\Tilde{x}\|^2-\|x^{k}-y^{k}\|^2-\|y^{k}-x^{k+1}\|^2\nonumber \\
        &+2\langle x^{k}-\alpha F(y^k)-y^{k},x^{k+1}-y^{k} \rangle.
    \end{align}
We can estimate the last term in~\eqref{eq-alg 3} in the following form using the Cauchy inequality
\begin{align}\label{eq-alg 4}
        &\langle x^{k}-\alpha F(y^k)-y^{k},x^{k+1}-y^{k} \rangle \nonumber \\
        =& \langle x^{k}-\alpha F(x^k)-y^{k},x^{k+1}-y^{k} \rangle \nonumber \\
        &+\alpha\langle  F(x^k)-F(y^k),x^{k+1}-y^{k} \rangle \nonumber \\
        \leq &
        \alpha \langle F(x^k)-F(y^k),x^{k+1}-y^{k} \rangle \nonumber \\
        \leq &\alpha\|F(x^k)-F(y^k)\|\|x^{k+1}-y^{k}\|,
    \end{align}
    where the first inequality is obtained using the fact that $\langle x^{k}-\alpha F(x^k)-y^{k},x^{k+1}-y^{k} \rangle \le 0$ which follows from the projection inequality  $\langle z-\Pi_K[z],x-\Pi_K[z]\rangle\le 0$ for all $z\in\mathbb{R}^m$ and $x\in K$, the definition of $y^k$, and $x^{k+1}\in K$. 
Combining~\eqref{eq-alg 3} and \eqref{eq-alg 4}, and using the Lipschitz continuity of the mapping $F$,  we obtain the following relation
\begin{align}\label{eq-alg 6}
        \|x^{k+1}-\Tilde{x}\|^2\leq&\|x^{k}-\Tilde{x}\|^2-\|x^{k}-y^{k}\|^2-\|y^{k}-x^{k+1}\|^2 \nonumber \\
        &+2\alpha L\|x^k-y^k\|\|x^{k+1}-y^{k}\| \nonumber \\
        \leq&\|x^{k}-\Tilde{x}\|^2-\|x^{k}-y^{k}\|^2-\|y^{k}-x^{k+1}\|^2 \nonumber \\
        &+\alpha^2 L^2\|x^k-y^k\|^2+\|x^{k+1}-y^{k}\|^2,
    \end{align}
    where in the last inequality in~\eqref{eq-alg 6} we use 
\[2\alpha  L\|x^k-y^k\|\|x^{k+1} - y^k\|
\le 
        \alpha^2 L^2\|x^k-y^k\|^2+\|x^{k+1} - y^k\|^2.\]
From~\eqref{eq-alg 6} it follows that for all $k\ge0$, 
\begin{align*}
%\label{eq-alg 7}
        \|x^{k+1}-\Tilde{x}\|^2&\leq\|x^{k}-\Tilde{x}\|^2-(1-\alpha^2 L^2)\|x^k-y^k\|^2.
    \end{align*}
Therefore, for $\alpha<\frac{1}{L}$, we see that $\|x^k-y^k\|\to0$ and $\|x^{k}-\Tilde{x}\|^2$ converges, and, hence, $\{x^{k}\}_{k=0}^\infty$ is bounded. As a result, for every convergent subsequence $\{x^{k_i}\}_{i=1}^\infty$ with $\lim_{i\rightarrow\infty}x^{k_i}=\hat{x}$, the limit point $\hat x$ is in the set $K$ since $K$ is closed.
From the definition of $y^k$, we have that 
\[\|y^k-x^k\|=\|\Pi_K[x^k-\alpha F(x^k)]-x^k\|.\]
Since $\|x^k-y^k\|\to0$, it follows that 
$\|\Pi_K[x^k-\alpha F(x^k)]-x^k\|\to 0$.
Thus, for any convergent subsequence $\{x^{k_i}\}_{i=1}^\infty$ with
$\lim_{i\rightarrow\infty}x^{k_i}=\hat{x}$, it follows that $\|\Pi_K[x^k-\alpha F(x^k)]-x^k\|=0,$
implying that 
$\hat{x}=\Pi_K[\hat{x}-\alpha F(\hat{x})]$. 
Hence, $F_K^{nat}(\hat x)=0$ and by Theorem~\ref{Thm-auxiliary-natural map solution}, we have that 
$\hat{x}\in {\rm SOL}(K,F)$.
%$\langle \hat{x}-\alpha F(\hat{x})-\hat{x},v-\hat{x}\rangle\leq 0$ for all $v\in K$, which means $\langle \alpha F(\hat{x}),v-\hat{x}\rangle\geq 0$. Therefore, 
\end{proof}
We note that the assumption of Theorem~\ref{Thm-extragradient main variant} that VI$(K,F)$ has a Minty solution together with the assumption that the mapping $F$ is Lipschitz continuous implies that VI$(K,F)$ has a solution
by Lemma~\ref{Lem-Minty Lemma}(a).

% Using the previous algorithm, since $\|x^{k}-\Tilde{x}\|$ forms a decreasing sequence which is bounded from below, it is convergent, and the limit is $\|\hat{x}-\Tilde{x}\|$. We can also simply show from this fact that $\|x^k-y^k\|$ converges to zero, $F_K^{nat}(x^k)\rightarrow 0$ as $k\rightarrow \infty$.

% \begin{assumption}\label{Assum-existence of minty solution natural map}
% For some $\Tilde{x}$ with $F_K^{nat}(\Tilde{x})=0$ there exists a strongly monotone map $\varphi$ such that $\varphi(\Tilde{x})=0$ and $\|\varphi(x)-F_K^{nat}(x)\|<d\|x-\Tilde{x}\|$ where $ d \leq \mu_{\varphi}$.
% \end{assumption}
% \begin{assumption}\label{Assum-Lipschitz continuity natural map}
% There exists a constant $L$ such that $\|F_K^{nat}(x)-F_K^{nat}(y)\|\leq L\|x-y\|$ for all $x,y\in \mathbb{R}^m$.
% \end{assumption} 

We next provide a condition under which the iterates of the extra-gradient method converge to a Minty solution.
\begin{assumption}\label{Assum-existence of minty solution}
Let $\varphi:K\to\mathbb{R}^m$ be a strongly monotone mapping with $\Tilde{x}\in {\rm SOL}(K,\varphi)$.
Assume that $\|\varphi(x)-F(x)\|\le d\|x-\Tilde{x}\|$ for some $d < \mu_{\varphi}$ and for all $x\in K$, where $\mu_{\varphi}$ is the strong monotonicity constant of the mapping $\varphi$.
%For some $\Tilde{x}$ with $F_K^{nat}(\Tilde{x})=0$ there exists a strongly monotone map $\varphi:K\to\mathbb{R}^m$ such that $\Tilde{x}\in {\rm SOL}(K,\varphi)$ and $\|\varphi(x)-F(x)\|<d\|x-\Tilde{x}\|$ where $ d \leq \mu_{\varphi}$.
\end{assumption}

The following theorem shows that, under the assumptions of this section and a suitable choice of the stepize, the extra-gradient method converges to a Minty solution of VI$(K,F)$.
\begin{theorem}\label{Thm-Koorpelevich with all assumptions}
    Let Assumptions~\ref{Assum-set-Lip} and~\ref{Assum-existence of minty solution} hold. Consider the method~\eqref{Alg-Korpelevich} with a stepsize $0<\alpha<\frac{1}{L}$. Then, $\{x^k\}_{k=0}^\infty$ and $\{y^k\}_{k=0}^\infty$ converge to the unique solution $\Tilde{x}\in {\rm SOL}(K,\varphi)$, which is a Minty solution to VI$(K,F)$.
\end{theorem}
\begin{proof}
Assumptions~\ref{Assum-set-Lip} and~\ref{Assum-existence of minty solution} imply that the unique solution $\tilde x$ to the VI$(K,\varphi)$ is a Minty solution to VI$(K,F)$ by Theorem~\ref{Thm-Minty sol}. 
Using relation~\eqref{eq-alg 1} with $x=\tilde x$, we obtain
\begin{align}\label{eq-alg 1'}
        \|x^{k+1}-\Tilde{x}\|^2\leq&\|x^{k}-\Tilde{x}\|^2-\|x^{k}-x^{k+1}\|^2 \nonumber \\
        &+2\alpha\langle F(y^k),\Tilde{x}-x^{k+1} \rangle.
    \end{align}
    By Assumption~\ref{Assum-existence of minty solution}, we have for all $x\in K$,
    $$
\langle\varphi(x)-F(x), x-\Tilde{x}\rangle 
\le\| \varphi(x)-F(x)\|\|x-\tilde x\|
\le d\|x-\Tilde{x}\|^2.
$$
    Since $\tilde x\in{\rm SOL}(K,\varphi)$, we have $\langle\varphi(\Tilde{x}),x-\Tilde{x}\rangle\geq 0$. Thus,
    we can write
    \begin{equation*}
        \langle\varphi(x), x-\Tilde{x}\rangle \leq \langle F(x), x-\Tilde{x}\rangle+\langle\varphi(\Tilde{x}),x-\Tilde{x}\rangle+d\|x-\Tilde{x}\|^2.
    \end{equation*}
   Re-arranging the terms in the preceding relation, we obtain
    \begin{equation}\label{eq-use of assumption-1}
        \langle\varphi(x)-\varphi(\Tilde{x}), x-\Tilde{x}\rangle-d\|x-\Tilde{x}\|^2 \leq \langle F(x), x-\Tilde{x}\rangle.
    \end{equation}
    Lower bounding the left hand side of \eqref{eq-use of assumption-1} by the strong monotonicity of $\varphi$ we have
    for all $x\in K,$
\begin{equation}\label{eq-use of assumption-2}
        (\mu_{\varphi}-d)\|x-\Tilde{x}\|^2 \leq \langle F(x), x-\Tilde{x}\rangle.
    \end{equation}
    Therefore, we have $\langle F(y^k),\Tilde{x}-y^{k} \rangle \leq -(\mu_{\varphi}-d)\|y^k-\Tilde{x}\|^2$.
    Now, we estimate the inner product in 
    \eqref{eq-alg 1'}, as follows:
    \begin{align}\label{eq-alg 2'}
        \langle F(y^k),\Tilde{x}-x^{k+1} \rangle=&\langle F(y^k),\Tilde{x}-y^{k} \rangle +\langle F(y^k),y^{k}-x^{k+1} \rangle \nonumber \\
        \leq & \langle F(y^k),y^{k}-x^{k+1} \rangle\cr
        &-(\mu_{\varphi}-d)\|y^k-\Tilde{x}\|^2.
    \end{align}
    Using \eqref{eq-alg 2'} in \eqref{eq-alg 1'} we obtain
     \begin{align}\label{eq-alg 3'}
        \|x^{k+1}-\Tilde{x}\|^2
        \leq&\|x^{k}-\Tilde{x}\|^2-\|x^{k}-x^{k+1}\|^2 \nonumber \\
        &+2\alpha\langle F(y^k),y^{k}-x^{k+1} \rangle-2\alpha(\mu_{\varphi}-d)\|y^k-\Tilde{x}\|^2 \nonumber \\
        =&\|x^{k}-\Tilde{x}\|^2-\|x^{k}-y^{k}\|^2-\|y^{k}-x^{k+1}\|^2 \nonumber \\
        &-2\langle x^{k}-y^{k},y^{k}-x^{k+1} \rangle+2\alpha\langle F(y^k),y^{k}-x^{k+1} \rangle \nonumber \\
        &-2\alpha(\mu_{\varphi}-d)\|y^k-\Tilde{x}\|^2\nonumber \\
        =&\|x^{k}-\Tilde{x}\|^2-\|x^{k}-y^{k}\|^2-\|y^{k}-x^{k+1}\|^2\nonumber \\
        &+2\langle x^{k}-\alpha F(y^k)-y^{k},x^{k+1}-y^{k} \rangle\nonumber \\
        &-2\alpha(\mu_{\varphi}-d)\|y^k-\Tilde{x}\|^2.
    \end{align}
Next, we estimate the inner product term in~\eqref{eq-alg 3'}, as follows
\begin{align}\label{eq-alg 4'}
        &\langle x^{k}-\alpha F(y^k)-y^{k},x^{k+1}-y^{k} \rangle \nonumber \\
        =&\langle x^{k}-\alpha F(x^k)-y^{k},x^{k+1}-y^{k} \rangle \nonumber \\
        &+\alpha\langle  F(x^k)-F(y^k),x^{k+1}-y^{k} \rangle \nonumber \\
        \leq &\alpha\langle  F(x^k)-F(y^k),x^{k+1}-y^{k} \rangle \nonumber \\
        \leq&\alpha\|F(x^k)-F(y^k)\|\|x^{k+1}-y^{k}\|,
    \end{align}
    where the first inequality follows from $\langle x^{k}-\alpha F(x^k)-y^{k},x^{k+1}-y^{k} \rangle \le 0$, which is due to the definition of $y^k$, the projection inequality  $\langle z-\Pi_K[z],x-\Pi_K[z]\rangle\le 0$ for all $z\in\mathbb{R}^m$ and $x\in K$, and the fact that $x^{k+1}\in K$.

Using the Lipschitz continuity of the mapping $F$, and combining~\eqref{eq-alg 3'} and \eqref{eq-alg 4'}, we have the following
relation
\begin{align}\label{eq-alg 6'}
        \|x^{k+1}-\Tilde{x}\|^2=&\|x^{k}-\Tilde{x}\|^2-\|x^{k}-y^{k}\|^2-\|y^{k}-x^{k+1}\|^2 \nonumber \\
        &+2\alpha L\|x^k-y^k\|\|x^{k+1}-y^{k}\| \nonumber \\
        &-2\alpha(\mu_{\varphi}-d)\|y^k-\Tilde{x}\|^2 \nonumber \\
        \leq&\|x^{k}-\Tilde{x}\|^2-\|x^{k}-y^{k}\|^2-\|y^{k}-x^{k+1}\|^2 \nonumber \\
        &+\alpha^2 L^2\|x^k-y^k\|^2+\|x^{k+1}-y^{k}\|^2\nonumber \\
        &-2\alpha(\mu_{\varphi}-d)\|y^k-\Tilde{x}\|^2,
    \end{align}
where the last inequality in~\eqref{eq-alg 6'} follows from
\[ 2\alpha L\|x^k-y^k\|\|x^{k+1}-y^{k}\| 
\le \alpha^2 L^2\|x^k-y^k\|^2+\|x^{k+1}-y^{k}\|^2. \]
Thus, from~\eqref{eq-alg 6'} we obtain 
\begin{align*}
\label{eq-alg 7'}
        \|x^{k+1}-\Tilde{x}\|^2\le &\|x^{k}-\Tilde{x}\|^2-(1-\alpha^2 L^2)\|x^k-y^k\|^2 \nonumber \\
        &-2\alpha(\mu_{\varphi}-d)\|y^k-\Tilde{x}\|^2.
    \end{align*}
    The preceding relation holds for all $k\ge0$.
Since $0<\alpha<\frac{1}{L}$ and $d<\mu_\varphi$, we can see that $\|x^k-y^k\|\to0$ and $\|y^k-\Tilde{x}\|\to0$, implying that
$\|x^{k}-\Tilde{x}\|^2\to0$. converges, $\|x^k-y^k\|$, Therefore, $\{x^k\}$ and $\{y^k\}$ converge to $\Tilde{x}$. 
\end{proof}

\section{Conclusions}\label{Sec-conclusion}
In this paper, we studied non-constrained non-monotone VIs through the inverse mapping theorem and obtained some conditions for the existence of solutions to such VIs. We showed that mappings with orthogonal structures or those mappings satisfying our weak coupling criteria meet this set of conditions, guaranteeing a solution exists to the VI. We stepped forward and extended these results for the case of constrained non-monotone VIs where the domain is nonempty closed convex but not necessarily compact. Moreover, we derived some conditions that guarantee there exists a Minty solution given the fact that at least one solution exists for a related VI. Finally, we showed that the extra-gradient method can converge to one of the solutions under some suitable assumptions. In the context of game theory, these results can be interpreted and lead to obtaining sufficient conditions to guarantee the existence of a quasi-Nash or Nash equilibrium in the corresponding games, which we will explore in the future. Hence, the extra gradient method can be efficient in obtaining a quasi-Nash or a Nash equilibrium for the class of games satisfying our conditions.

\bibliographystyle{plain}
\bibliography{main}
\end{document}